\documentclass[twoside,leqno,symbols-for-thanks]{amsart}
\usepackage[colorlinks,linkcolor=black,citecolor=black,urlcolor=black, linktocpage=true]{hyperref}
\usepackage{srcltx}

\usepackage[english]{babel}
\numberwithin{equation}{section}
\usepackage{graphicx,amssymb,latexsym,amsfonts,amsmath,amsthm}
\usepackage{pdfsync,color,tabularx,rotating}
\usepackage[all,cmtip]{xy}
\usepackage{tikz}

\newtheorem{theo}{Theorem}
\newtheorem{prop}[theo]{Proposition}
\newtheorem{coro}[theo]{Corollary}
\newtheorem{lemm}[theo]{Lemma}

\theoremstyle{remark}
\newtheorem{rema}[theo]{\bf Remark}
\newtheorem{exam}[theo]{\bf Example}
\newtheorem{defi}[theo]{\bf Definition}

\title{Zapponi-orientable dessins d'enfants}

\author[E. Girondo, G. Gonz\'alez-Diez, R. A. Hidalgo, G. A. Jones]{Ernesto Girondo, Gabino Gonz\'alez-Diez, Rub\'en A. Hidalgo\\ and Gareth A. Jones}

\address[ernesto.girondo@uam.es]{{\sc Ernesto Girondo}: Departamento de Matem\'aticas, UAM and ICMAT, Madrid (Spain)}

\address[gabino.gonzalez@uam.es]{{\sc Gabino Gonz\'alez-Diez}: Departamento de Matem\'aticas, UAM, Madrid (Spain)}

\address[ruben.hidalgo@ufrontera.cl]{{\sc Ruben A. Hidalgo}: Departamento de Matem\'atica y Estad\'{\i}stica, Universidad de La Frontera.  Temuco (Chile)}

\address[G.A.Jones@maths.soton.ac.uk]{{\sc Gareth A. Jones}: School of Mathematics, University of Southampton, Southampton SO17 1BJ (UK)}

\subjclass[05C10, 05C25, 30F10]{14H57}


\begin{document}


\begin{abstract}
Almost two decades ago Zapponi introduced a notion of orientability of a clean dessin d'enfant, based on an orientation of the embedded bipartite graph. We extend this concept, which we call Z-orientability to distinguish it from the traditional topological definition, to the wider context of all dessins, and we use it to define a concept of twist orientability, which also takes account of the Z-orientability properties of those dessins obtained by permuting the roles of white and black vertices and face-centres. We observe that these properties are Galois-invariant, and we study the extent to which they are determined by the standard invariants such as the passport and the monodromy and automorphism groups. We find that in general they are independent of 
these invariants, but in the case of regular dessins they are determined by the monodromy group.
\end{abstract}


\maketitle

\section{Introduction}

According to a famous theorem of Belyi~\cite{Belyi}, a compact Riemann surface $S$, regarded as a complex projective algebraic curve, can be defined over the field $\overline{\mathbb Q}$ of algebraic numbers if and only if it admits a non-constant meromorphic function $\beta:S\to\widehat{\mathbb C}={\mathbb C}\cup\{\infty\}\simeq{\mathbb P}^1({\mathbb C})$ which is branched over at most three points (which we can take to be $0, 1$ and $\infty$). Such a function $\beta$ is called a {\em Belyi function}, and $(S,\beta)$ is known as a {\em Belyi pair}. The `certificate' for possessing such a function is the existence of a map embedding a bipartite graph in $S$, with the fibres $\beta^{-1}(0)$ and $\beta^{-1}(1)$ providing the white and black vertices, and the unit interval $I=[0,1]\subset{\mathbb R}$ lifting to the edges (the complement of the embedded graph is a disjoint union of discs, called the {\it faces} of the map). Grothendieck~\cite{Gro} called such a map $\mathcal D$ a {\em dessin d'enfant}, or simply a {\em dessin}, and noted that the natural action of the absolute Galois group ${\rm Gal}(\overline{\mathbb Q}/{\mathbb Q})$ on Belyi pairs induces a faithful action on (isomorphism classes of) dessins.

One of the major problems in this area is that of determining Galois orbits, or more specifically of determining whether or not two given dessins are in the same orbit of ${\rm Gal}(\overline{\mathbb Q}/{\mathbb Q})$. If they differ in some Galois-invariant property, such as their genus, type, monodromy group, automorphism group or passport, then we know that they are in different orbits. However, if they agree in these properties then nothing can be concluded, and we need to try further tests. For this reason it is useful to know and to be able to determine as many Galois invariants as possible, in order to have a more powerful toolkit in tackling this problem. Of course, a new invariant needs to be independent of the others (for instance, the passport determines the genus and the type, but not the monodromy or automorphism group), or to be more easily computed, in order to be of any use.

In~\cite{ Zapponi0, Zapponi} Zapponi introduced a new Galois invariant, which he called `orientability' (motivated by the graph-theoretic concept of an orientation of the embedded graph). Of course, the term `orientability' has a traditional meaning in topology, and in that sense all dessins are orientable since they lie on Riemann surfaces, so to avoid confusion we will use the term `Zapponi-orientability' for his concept. By abbreviating it to `Z-orientability' we will commemorate both his name and the special role played by the third standard generator $Z$ of a triangle group in its definition. Thus a dessin is said to be {\em Z-orientable} if its faces can be assigned labels $+$ or $-$ so that each edge is incident with faces having different labels. (In graph theory this is known as a 2-face colouring, so that the dual map is bipartite.) We will show that this property can be reinterpreted in a number of different ways, so that it can be studied in terms of function theory, group theory or combinatorics. One of our main results is as follows, where $M$ denotes the monodromy group of $\mathcal D$ (that is, of the covering $\beta:S\to\widehat{\mathbb C}$) and $\sigma_0$, $\sigma_1$ and $\sigma_{\infty}$, with $\sigma_0\sigma_1\sigma_{\infty}=1$, are its standard generators, giving the monodromy permutations of the edges (equivalently, the sheets of the covering) at the critical values.

\medskip

\begin{theo}\label{mainthm}
Each of the following properties of a dessin $\mathcal D$ is equivalent to its Z-orientability:
\begin{enumerate}
\item The function $\beta(1-\beta)$ is the square of a meromorphic function on $S$.
\item The edges of $\mathcal D$ can be labelled $+$ or $-$ so that successive edges around each vertex have different signs.
\item $M$ acts imprimitively on the edges, with two blocks transposed by $\sigma_0$ and $\sigma_1$.
\item $M$ has a subgroup of index $2$ containing $\sigma_{\infty}$ and the stabiliser of each edge.
\item $\mathcal D$ covers the unique dessin of degree $2$ and type $(2,2,1)$.
\end{enumerate}
\end{theo}

\medskip

Part (1) of Theorem \ref{mainthm} will be proved in Section \ref{Sec:orientable}, whereas the remaining parts will be proved in Section \ref{alternative}. It follows immediately from part (1) of the above theorem that the Z-orientability or otherwise of a dessin $\mathcal D$ is a Galois invariant. As we will show by example, it is independent of the more obvious and widely used Galois invariants such as the monodromy group, automorphism group and passport, so it represents an additional tool which can be used in trying to distinguish Galois orbits.

The traditional choice of using the fibres over $0$ and $1$, but not $\infty$, as vertices of the dessin is often convenient, but it has the disadvantage of disguising the complete symmetry of the situation. To rectify this, we will also consider the dessins ${\mathcal D}'$ and ${\mathcal D}''$ which can be obtained from $\mathcal D$ by cyclically permuting the roles of the three critical values of $\beta$ (of course, simply transposing vertex colours has no effect on Z-orientability). There are obvious analogues of Theorem~\ref{mainthm} for these `twisted' dessins, representing a further refinement of Z-orientability as a useful Galois invariant. We will introduce a new Galois invariant, the {\em twist-invariant orientability type} ${\rm tot}({\mathcal D})$, which is the number of Z-orientable dessins among the triple  $\{{\mathcal D}, {\mathcal D}', {\mathcal D}''\}$, and we will show in Proposition~\ref{totinvariant} that it can take the values $0, 1$ or $3$, but never $2$.

In order to show that the Z-orientability properties of a dessin are independent of the more traditional Galois invariants mentioned above, in our final section we will construct, for each integer $n\ge 8$, four uniform dessins (see Section \ref{Sec:dessins} for the definition of uniform dessin) ${\mathcal D}_i\;(i=0,1,2, 3)$ with the same degree $N=(n!/4!)^2$, the same type, the same monodromy group $M = S_n\times S_n$
(as an abstract group) and the same passport; these dessins also have isomorphic (but non-conjugate) edge-stabilisers and isomorphic automorphism groups. However, we will also show that ${\rm tot}({\mathcal D}_{0})=0$, ${\rm tot}({\mathcal D}_{1})={\rm tot}({\mathcal D}_{2})=1$, ${\rm tot}({\mathcal D}_{3})=3$, and that ${\mathcal D}_2$ is Z-orientable whereas ${\mathcal D}_1$ is not, so the four dessins ${\mathcal D}_i$ lie in distinct Galois orbits. Thus Z-orientability properties can distinguish between these orbits, whereas the other Galois invariants listed above can not. In particular, this example proves the following:

\medskip
\begin{theo}\label{independent}
The Galois invariant $\rm tot$ of a dessin is independent of its monodromy group, point stabilisers and automorphism group (as abstract groups) and of its passport.
\end{theo}

\medskip

However, in Section \ref{tot} we will prove the following (see Corollary \ref{teo16}):

\medskip
\begin{theo}
For regular dessins the invariant ${\rm tot}$ is determined by the monodromy group (as an abstract group).
\end{theo}


\section{Preliminaries and notations}\label{Sec:dessins}

Recall that a \emph{dessin d'enfant} is a connected bipartite graph $\mathcal{D}$ embedded in a closed oriented surface $X$ in such a way that the connected components of $ X \smallsetminus \mathcal{D}$, which are called \emph{faces}, are homeomorphic to open discs. The vertices of a dessin can be coloured white or black in such a way that the two  vertices of any edge do not have the same colour. We say that a face has degree $r$ if it is topologically a polygon of $2r$ sides, while the degree of a vertex is simply the number of edges incident with it. Two dessins d'enfants are called {\it equivalent} if there is an orientation-preserving homeomorphism between the closed orientable surfaces inducing a colour-preserving isomorphism of the corresponding bipartite graphs.

The \emph{passport} of a dessin ${\mathcal D} \subset X$ is the tuple
$(a_{1},\ldots,a_{\alpha};b_{1},\ldots,b_{\beta};c_{1},\ldots,c_{\gamma}),$
where $1\leq a_{1}\leq \cdots \leq a_{\alpha}, \; 1\leq b_{1}\leq \cdots \leq b_{\beta}, \; 1\leq c_{1}\leq \cdots \leq c_{\gamma}$ and 
 ${a_i}$, ${b_j}$ and ${c_k}$ are the degrees of the white vertices, black vertices and faces, respectively.  
 If the least common multiples of the degrees of the white vertices, black vertices and faces are $a, b$ and $c$ respectively,
 then we say that ${\mathcal D}$ has \emph{type} $(a,b,c)$.  
If $\mathcal{D}$ has $n$ edges, then 
$n=a_{1}+\cdots+a_{\alpha}=b_{1}+\cdots+b_{\beta}=c_{1}+\cdots+c_{\gamma},$
and, as a consequence of Euler's formula, the genus of $X$ is
$g=1+\frac{1}{2} \left( n-\alpha-\beta-\gamma\right)$.

After labelling the edges of $\mathcal{D}$ from $1$ to $n$, the orientation of $X$ induces  permutations $\sigma_0, \sigma_1$ encoding the cyclic ordering of the different edges around the white  or black vertices respectively.  The pair $\sigma_0, \sigma_1$ is called the \emph{permutation representation} of  $\mathcal{D}$, and the group $M= \langle \sigma_0, \sigma_1 \rangle$, which is transitive since $\mathcal{D}$ is connected, is called the \emph{monodromy group} of the dessin.

Conversely, any pair of permutations $\sigma_0, \sigma_1$ generating a transitive subgroup of the symmetric group $S_n$ arises as the permutation representation of a dessin d'enfant $\mathcal{D}$ with $n$ edges. The disjoint cycles of $\sigma_0$, $\sigma_1$ and $\sigma_{\infty}:=(\sigma_0 \sigma_1)^{-1}$ are in bijective correspondence with  the white vertices, black vertices and faces of $\mathcal{D}$, and their lengths encode its passport.

A Belyi pair $(S, \beta)$ consists of a compact Riemann surface $S$ (called a Belyi surface) and a non-constant meromorphic function $\beta$ on $S$ (called a Belyi function) with at most three branching values which can be assumed to be $0$, $1$ and $\infty$. Two Belyi pairs $(S_{1},\beta_{1})$ and  $(S_{2},\beta_{2})$ are called {\it equivalent} if there is a bi-holomorphism $\psi:S_{1} \to S_{2}$  such that $\beta_{2} \circ \psi= \beta_{1}$.
There is a well-known bijective correspondence between
the following classes of objects (see e.g. \cite{GiGo}, \cite{JW}): 

\begin{enumerate}
\item  Equivalence classes of dessins.
\item Equivalence classes of Belyi pairs.
\item  Conjugacy classes of finite index subgroups of  triangle groups $\Delta(a,b,c)$.
\item  Conjugacy classes of finite index subgroups of  the free group $F_2$ of rank two.
\item Isomorphism classes of two-generator transitive finite permutation groups.  
\end{enumerate}

The link between these  classes of objects is made as follows.

(1) $\leftrightarrow$ (2) : Given a Belyi pair $(S, \beta)$ one gets a dessin d'enfant by setting
$X = S$, $\mathcal{D} = \beta^{-1} (I)$, where $I$ stands for the unit interval $[0,1]$ in
$\widehat{\mathbb{C}}$, and declaring white (resp.~black) vertices to be the points in the fibre over $0$ (resp.~$1$).  Conversely,  given a dessin $\mathcal{D} \subset X$ one can define a Riemann surface structure on the topological surface $X$ and a Belyi function $\beta$ on it such that $\mathcal{D}=f^{-1}(I)$, with poles corresponding to the centres of faces.

(2) $\leftrightarrow$ (4) :  Given a Belyi pair $(S,\beta)$, we may consider the smooth covering $\beta:S^{*} \to \widehat{\mathbb C}\setminus\{0,1,\infty\}$, where $S^{*}$ is the complement in $S$ of the $\beta$-preimages of $0$, $1$ and $\infty$. By standard covering space theory, these coverings are in bijective correspondence with the finite index subgroups of the fundamental group of $\widehat{\mathbb C}\setminus\{0,1,\infty\}$, which is isomorphic to the free group of rank two.  

(1) $\leftrightarrow$ (3) : The permutation representation of a dessin $\mathcal{D} \subset X$ of type $(a,b,c)$ and $n$ edges induces  a group homomorphism $\omega$, called the \emph{monodromy homomorphism} of the dessin, 
\[\omega: \Delta(a,b,c)= \langle x, y, z \  | \ x^a=y^b=z^c =xyz=1 \rangle \longrightarrow S_n,\] 
determined by $\omega(x) = \sigma_0$, $\omega(y) = \sigma_1$ and $\omega(z)=\sigma_{\infty}$, whose image is the monodromy group $M$. Denoting by $\mathrm{Stab}_M(1)$ the stabiliser subgroup of 1 in $M$ and by  $\Gamma$ its preimage in the triangle group $\Delta=\Delta(a,b,c)$, it can be easily checked that the inclusion $\Gamma \le \Delta$ induces a Belyi map $ \beta:\mathbb{D}/ \Gamma \to \mathbb{D}/ \Delta$
 on $S= \mathbb{D}/ \Gamma$, where $\mathbb D$ is the unit disc; the corresponding dessin is precisely $\mathcal{D}$ and the corresponding monodromy homomorphism 
\[ \omega_{\beta}: \pi_1(\widehat{\mathbb{C}} \smallsetminus \{ 0,1, \infty \}) \longrightarrow S_n\]
 is obtained from $\omega$ by pre-composition with the obvious homomorphism 
\[\pi_1(\widehat{\mathbb{C}} \smallsetminus \{ 0, 1, \infty \}) = \Delta(\infty, \infty, \infty) \longrightarrow \Delta(a,b,c).\]
 
(4) $\leftrightarrow$ (5) :  Given a subgroup of finite index in $F_2$, we obtain a two-generator transitive finite permutation group by letting $F_2$ act on its cosets by right multiplication, and factoring out the kernel of the action. Conversely, any such permutation group is a quotient of $F_2$, acting on the cosets of a subgroup of finite index, which is unique up to conjugacy.
 
It is often useful to reinterpret (5) in the following way, in order to distinguish between the monodromy group of a dessin as an {\em abstract\/} group, and as a {\em permutation\/} group.  We consider $5$-tuples $(M, H, m_{0}, m_{1},m_{\infty})$, where $M$ is a finite group generated by elements $m_{0}$ and $m_{1}$, $m_{0}m_{1}m_{\infty}=1$, and $H$ is a subgroup of $M$ such that its core $\cap_{m \in M} m^{-1}Hm$ in $M$ is the identity subgroup. This $5$-tuple corresponds to the dessin $\mathcal D$ associated with the embedding $\varphi:M \to {\rm Sym}(M/H) \cong S_{n}$ where $n$ is the index of $H$ in $M$, and $M$ acts on cosets of $H$ in the usual way. In this case, $\varphi(M)$ is the monodromy group of $\mathcal D$ and $\varphi(H)$ is the stabiliser of an edge, with $\sigma_{0}=\varphi(m_{0})$, $\sigma_{1}=\varphi(m_{1})$ and $\sigma_{\infty}=\varphi(m_{\infty})$. Conversely, any dessin determines such a $5$-tuple, where $M$ is its monodromy group and $H$ is the subgroup stabilising an edge. In this setting, equivalence of dessins corresponds to isomorphism of $5$-tuples $(M, H, m_{0}, m_{1},m_{\infty})$ and $(M', H', m'_{0}, m'_{1},m'_{\infty})$, meaning an isomorphism $\psi:M \to M'$ such that $\psi(m_{0})=m'_{0}$, $\psi(m_{1})=m'_{1}$, $\psi(m_{\infty})=m'_{\infty}$ and $\psi(H)$ is conjugate to $H'$.

The group $\Gamma$ is torsion-free exactly when $a_i=a$, $b_j=b$ and $c_k=c$ for all $i, j$ and $k$ respectively. Such dessins are called  \emph{uniform dessins}. \emph{Regular dessins} are a subclass of uniform dessins characterized by any of the following equivalent conditions: i) $\mathrm{Stab}_M(1)$ is trivial, ii) $\Gamma$ is normal in $\Delta$, iii) $\beta$ is a Galois (or normal or regular) covering, and iv) the group of orientation-preserving automorphisms of $\mathcal{D}$ as an embedded bicolored graph  acts transitively on the set of edges of the dessin.

Clearly, the roles played by the points $0, 1$ and $\infty$ in the above discussion can be interchanged in order to produce dessins related to each other in a  natural way. For instance, if $\beta$ is a Belyi function with dessin $\mathcal{D}$ then $1-\beta$ is a Belyi function whose dessin is obtained from $\mathcal{D}$ by interchanging the colours of the vertices. More generally, something similar happens by post-composition of $\beta$ with {\em any\/} element of the order 6 group of M\"obius transformations preserving  the set $\{0, 1, \infty\}$. We will say that dessins obtained in this way are \emph{twist-related} to each other. Note that twist-related dessins have the same monodromy group, and that their passports are related by permutations of the branching indices.

For example, let $\mathcal{D}'$ be the dessin obtained from $\mathcal D$ by replacing $\beta$ with the Belyi function $(\beta-1)/\beta$, so that the white and black vertices and face centres of $\mathcal D$ become the face centres and white and black vertices of ${\mathcal D}'$; this operation cyclically permutes their roles by replacing the triple $(\sigma_0, \sigma_1, \sigma_{\infty})$ with the triple $(\sigma_1, \sigma_{\infty}, \sigma_0)$. Repeating (or inverting) this operation gives a dessin $\mathcal{D}''$ corresponding to the Belyi function $1/(1-\beta)$ and the triple $(\sigma_{\infty}, \sigma_0, \sigma_1)$; the white and black vertices and face centres of $\mathcal D$ are the black vertices, face centres and white vertices of ${\mathcal D}''$.

A famous result by Belyi \cite{Belyi} shows that a Riemann surface admits an algebraic model with coefficients in $\overline{\mathbb Q}$ if and only if it admits a Belyi function, which in turn can also be defined over the algebraic numbers. The action of the absolute Galois group ${\rm Gal}(\overline{\mathbb Q}/{\mathbb Q})$ on the coeficients of  Belyi pairs induces an action   on the collection of dessins d'enfants. If ${\mathcal D}$ is a dessin and $\tau \in {\rm Gal}(\overline{\mathbb Q}/{\mathbb Q})$, then ${\mathcal D}^{\tau}$  denotes the image of  ${\mathcal D}$ under the action of $\tau$.  Some basic  invariants of the Galois  action  are the passport, the genus, the number of edges, and the isomorphism classes of the monodromy  and  automorphism groups \cite{GiGo,JS,JW}. 


\section{Z-orientable dessins d'enfants}\label{Sec:orientable}

Following Zapponi \cite{Zapponi0, Zapponi} we shall say that a dessin d'enfant is \emph{Zapponi-orientable} (abbreviated here to \emph{Z-orientable}) if its faces can be labelled with symbols $+$ and $-$ in such a way that each edge meets two faces of different sign. Z-orientability imposes the obvious necessary condition that white and black vertices must have even degree. 
Note also that according to this definition a dessin with \emph{monofacial} edges, meaning  edges meeting the same face on both sides (e.g{\color{red}.~}a dessin with only one face) is never Z-orientable. 

\medskip
\begin{theo} \label{th:squarecharact}
Let $\mathcal{D}$ be a dessin with associated Belyi pair $(S, \beta)$. Then $\mathcal{D}$ is Z-orientable if and only if $\beta(1-\beta)$ is the square of a meromorphic function on $S$.
\end{theo}
\begin{proof}
We begin  with the observation that $\beta(1-\beta)$ is the square of a meromorphic function if and only if 
the quadratic differential
\[\Phi = \displaystyle\frac{d \beta^{\otimes 2}}{\beta(1-\beta)}\]
is the square of a meromorphic $1$-form. This is because any  meromorphic $1$-form  can be written as $d\beta/h$ for a suitable meromorphic function $h$ (see e.g. \cite{GiGo}, Proposition 1.36).

So, suppose that $\Phi=\eta^{\otimes 2}$ for some meromorphic 1-form $\eta$. Let $F$ be a face 
of $\mathcal{D}$   and $p \in S$ the pole of $\beta$ lying inside  $F$. Suppose that the multiplicity of $p$ is $m$ so that the number of the edges $e_k$ of $F$ is $2m$, where monofacial edges are counted twice. An easy calculation in local coordinates shows that the residue of  $\eta$ at $p$,  
$  \mathrm{res}_p \eta,$ has to be either $ im$ or $-im$. Now we can write this residue $\pm im$ as
\[\displaystyle\frac{1}{2\pi i}\int_{\gamma} \eta=\displaystyle\frac{1}{2\pi i}\int_{\partial F} \eta = 
\displaystyle\frac{1}{2\pi i}  \displaystyle\sum_{k=1}^{2m} \int_{e_k} \eta = \displaystyle\frac{1}{2\pi i}  \displaystyle\sum_{k=1}^{2m} \int_{0}^1 \displaystyle\frac{dx}{\pm \sqrt{x(1-x)}} = \displaystyle\frac{1}{2\pi i}  \displaystyle\sum_{k=1}^{2m} \pi \varepsilon_k,
\]
where 
\begin{itemize}
\item $\varepsilon_k = \pm 1$,
\item $\gamma\subset F$ stands for any Jordan curve encircling  $p$ traversed in the same direction as $\partial F$ (the oriented boundary of $F$), 
\item  the edges $e_k$ of $F$ are assumed to be endowed with   the orientation induced by  $\partial F$ so in particular
each monofacial edge occurs twice, once for each  orientation (so that its contributions to the above sum cancel out),
\item the first equality holds because 
$\partial F$ and $\gamma$ are homology equivalent loops in the punctured Riemann surface $S\setminus \left\{ \text{poles of }
 \eta   \right\}$, and
\item the last equality holds because  $ \int \frac{dx}{ \sqrt{x(1-x)}} =  \arcsin (2x-1).$ 
	\end{itemize}
	We notice that  the identity \
$\pm im   =  \frac{1}{2\pi i}   \sum_{k=1}^{2m} \pi \varepsilon_k$ \
implies that the integrals $\int_{e_k} \eta= 
\pm \int_0^1 \frac{dx}{ \sqrt{x(1-x)}}$ 
 must all take the same value, namely  $\pi$ when  $\mbox{res}_p \eta = -im$ and $-\pi$  when $\mbox{res}_p \eta = im.$ In particular monofacial edges do not  occur when $\beta(1-\beta)$ is a square.
\newline
We can now 
Z-orient $\mathcal{D}$ 
by declaring  $F$  positive or negative according to whether  $\mbox{res}_p \eta = -im$ or $im$. 
By the previous observation this is equivalent to saying that the sign of $F$ is positive (resp. negative)
  when the integral of $\eta$ along any of its  (oriented) edges  equals $\pi$ (resp. $-\pi$). Of course when an edge is oriented according to the adjacent face the integral takes the opposite value, hence our rule assigns  opposite signs to adjacent faces, as it should.

Conversely, let us now assume that $\mathcal{D}$ is Z-oriented. We want to show that a   choice of $\sqrt{\beta(1-\beta)}$ can be made around each point so that the expression  
  \begin{equation} \label{1-form}
 \displaystyle\frac{d \beta}{\sqrt{\beta(1-\beta)}}
\end{equation} 
yields a globally defined  meromorphic $1$-form. \newline
As noted at the beginning of this section the zeroes as well as the preimages of $1$ under $\beta$ all have even degree. This implies that 
 the zeroes and poles of
$\beta(1-\beta)$ are also of even degree and therefore there is a well-defined
 square root of this function (hence, exactly two)
 on any simply connected  domain $U$. 

\medskip 
 \noindent
1) (\emph{Choice of $\sqrt{\beta(1-\beta)}$ on the interior of a face}) Let $F$ be a face 
of $\mathcal{D}$, $p \in S$ the pole of $\beta$ in $F$ and $m$ its multiplicity.  
Let  $U$ be the   interior of $F$.
Since $U$ is an open cell   there are two 
  well-defined holomorphic 
	square roots of
$\beta(1-\beta)$ on it.  
If $F$ has  positive (resp. negative) sign  we choose the one for which the residue of
the   1-form
on $U$  defined by the expression (\ref{1-form}) is
		$-im$ (resp. $im$). We shall denote this 1-form by $\eta$.
		
\medskip
\noindent 
2) (\emph{Choice of  $\sqrt{\beta(1-\beta)}$ around the edges})  We now need to make a choice of $\sqrt{\beta(1-\beta)}$  on a small simply connected  neighbourhood  
 $U_k$ of each of the edges $e_k$ of a given face $F$. This is equivalent to making a choice on  each $U_k$ 
of the  1-form  defined by (\ref{1-form}). If $F$ has positive (resp. negative) sign we choose the 1-form   $\eta_k$ that satisfies
$  \int_{e_k} \eta_k   =    \pi $ (resp.  $-\pi $)
with respect to  the orientation  on $e_k$ induced by the orientation of $\partial F$.
Notice that this choice is independent of which of the two faces incident to $e_k$ we look at because if $e_k$ is oriented according to the   other face $G$   
the integral
$  \int_{e_k} \eta_k  $ will take the opposite value. Since the sign  of $G$ is the opposite to the sign  of $F$ we see that 
   $\eta_k$ is also our choice from the perspective of the adjacent face $G.$ We recall that, as was observed in the introduction, being Z-orientable implies the absence of monofacial  edges.
  \newline
It remains to see that $\eta=\eta_k$ on each intersection $U \cap U_k$. 
 Suppose, for instance, that
  $F$ is positively oriented. 	In this case the choice of 
  $\eta$ we have  made is   the one fulfilling the identity 
\[ -im=\mathrm{res}_p \eta =\displaystyle\frac{1}{2\pi i}\int_{\gamma} \eta
\]
where $\gamma$ is any 
Jordan curve  around the pole $p$ lying  inside the interior $U\subset F$. Since $\mathcal{D}$ does not  have monofacial edges, the number of edges of $F$ is exactly $2m$. Now let
$\gamma$ 
approach the boundary $\partial F$ by means of  Jordan curves consisting of $2m$ segments
$\gamma_k$, each of which approaches the  corresponding edge $e_k$.
Then, for each $k$ the  integral $\int_{\gamma_k} \eta$ will get as close  to one of the values 
$\pm \int \frac{dx}{ \sqrt{x(1-x)}} =  \pm \pi$ 
as we wish, the sign 
depending on whether  $\eta=\eta_k$ or $\eta=-\eta_k$ on $U \cap U_k$. 
But then 
 one sees  that in order for
the identity
\[-im =\displaystyle\frac{1}{2\pi i}\int_{\gamma} \eta= 
\displaystyle\frac{1}{2\pi i}  \displaystyle\sum_{k=1}^{2m} \int_{\gamma_k} \eta 
\]
to hold each integral $\int_{\gamma_k} \eta $ must be close to $\pi$ and so  $ \eta $ must agree with $\eta_k$ on $U \cap U_k$. 
\end{proof} 

\medskip

Since being a square is preserved under the Galois action we deduce  the following:

\medskip
\begin{coro} \label{coro:orient}
Z-orientability is a Galois invariant.
\end{coro}

\medskip

It is well known that the action of the absolute Galois group on Belyi surfaces of each genus $g \geq 1$ is faithful, that is, for each non-identity $\tau \in {\rm Gal}(\overline{\mathbb Q}/{\mathbb Q})$ there is a Belyi surface  $S$ (in each genus $g \geq 1$, see~\cite{GiGo1}) so that $S^{\tau}$ is not isomorphic to $S$.  In fact, the Galois action is faithful even when restricted to \emph{quasiplatonic} Belyi surfaces (those admitting a regular Belyi function), see \cite{GoJa}.

\medskip
\begin{prop}
Every Belyi surface admits a Z-orientable dessin. In particular, the absolute Galois group acts faithfully on the class of Z-orientable dessins.
\end{prop}
\begin{proof}
 Let $\beta$ be a Belyi function on $S$, and consider the multiplicative inverse of the Klein $j$-modular function $t(x)=27x^{2}(x-1)^{2}/4(x^{2}-x+1)^3$. Both $t$ and 
 $t \circ \beta$ are Belyi functions \cite{GiGo} and one has the identity 
$$t(\beta(z))(1-t(\beta(z))=
\left(\frac{3\sqrt{3} (\beta(z)-2) (\beta(z)-1) \beta(z) (\beta(z)+1) (2 \beta(z)-1)}{4 \left(\beta(z)^2-\beta(z)+1\right)^3}\right)^{2}.$$
So, by Theorem \ref{th:squarecharact}, the Belyi function $t \circ \beta$ on $S$ defines a Z-orientable dessin. 
\end{proof}

\medskip

We next present a useful characterization of Z-orientability of dessins in terms of their permutation representation. This is a criterion established by Zapponi in \cite{Zapponi} for the so-called clean dessins, 
but one can easily see that it works for general dessins.

\medskip
\begin{prop}[Zapponi's criterion] \label{prop_zapponi}
Let $\mathcal{D}$ be a dessin with permutation representation $(\sigma_0, \sigma_1)\in S_n \times S_n$. Then $\mathcal{D}$ is Z-orientable if and only if the rule
$$
\sigma_0, \sigma_1 \longmapsto -1
$$
defines an epimorphism
$$
\rho: M=\langle \sigma_0, \sigma_1 \rangle \longrightarrow \{ \pm 1 \}
$$
such that the stabiliser subgroup $\mathrm{Stab}_M(1)$ is contained in $\ker(\rho)$.
\end{prop}
\begin{proof}
Let $\mathcal{D}$ be a Z-orientable dessin and $(S, \beta)$ its corresponding Belyi pair, so that $\beta(1- \beta)= h^2$ for some meromorphic function $h:S\to\widehat{\mathbb C}$. 

Consider the rational function  
$$
z  \longmapsto  - \displaystyle\frac{1}{4}\left( z - \displaystyle\frac{1}{z} \right)^2,
$$
which is a  branched Galois cover $\widehat{\mathbb{C}} \to \widehat{\mathbb{C}}$ whose deck transformation group is $\langle A_1 , A_2 \rangle \cong  {\mathbb Z}_{2}\times   {\mathbb Z}_{2} $, where $A_1(z)= 1/z$ and $A_2(z)=-1/z$. 

The  intermediate degree 2 branched Galois covers with 
deck group $\langle A_1 \rangle $ and $\langle A_2 \rangle $ are given by $\pi_{1}(z)=(z+1)^{2}/4z$ and $\pi_{2}(z)=(z^{2}-1)/4iz$ (we may observe that $\pi_{1}(z)=\pi_{2}(iz)+1/2$).
Therefore, setting 
$$
Q_1(z)=4z(1-z), \ Q_2(z)=4z^2,  \ \pi_1(z)=\displaystyle\frac{(z+1)^2}{4z}, \ \mbox{and } \pi_2(z)=\displaystyle\frac{(z^2-1)}{4iz}
$$
one has the following identities:
$$
\begin{array}{rcl}
Q_1 \circ \beta & = & Q_2 \circ h \\
Q_1 \circ \pi_1 & = & Q_2 \circ \pi_2 .
\end{array}
$$

From these two identities it is easy to see that the fiber product of the pairs $(\widehat{\mathbb{C}}, Q_1)$ and $(\widehat{\mathbb{C}}, Q_2)$ is the triple $(\widehat{\mathbb{C}}, \pi_1, \pi_2)$. 

The fiber product enjoys the following universal property, written for our situation (see Example 2.25 in \cite[page 30]{Ha}): given two morphisms $\phi_{1},\phi_{2}:S \to  \widehat{\mathbb C}$ such that  $Q_{1} \circ \phi_{1}=Q_{2} \circ \phi_{2}$, there exists a morphism $\psi:S \to \widehat{\mathbb C}$ such that $\phi_{j}=\pi_{j} \circ \psi$. So, by taking $\phi_{1}=\beta$ and $\phi_{2}=h$, this asserts 
 the existence of a morphism $P: S \to \widehat{\mathbb{C}}$ making the following diagram commutative

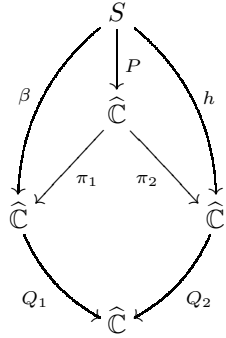
\centerline{
\xymatrix{ & 
S \ar@/_1pc/[ldd]_{\beta} \ar[d]^P \ar@/^1pc/[rdd]^h &  \\ & 
\widehat{\mathbb{C}} \ar[dl]^{\pi_1} \ar[dr]_{\pi_2}  & \\  \widehat{\mathbb{C}} \ar@/_/[dr]_{Q_1} & & \widehat{\mathbb{C}}  \ar@/^/[dl]^{Q_2} \\ &   \widehat{\mathbb{C}} & }
}

\

In particular we obtain a factorization of $\beta$ as $\beta = \pi_1 \circ P$. Now if
$$
\omega : \Delta(a,b,c) \longrightarrow M \le S_n
$$
is the monodromy homomorphism of the dessin, so that the inclusion of $\Gamma = \omega^{-1}(\mathrm{Stab}_M(1))$ in $\Delta=\Delta(a,b,c)$ induces the Belyi function $\beta$,  then the above factorization of $\beta$ corresponds to an intermediate subgroup $\Gamma < K \unlhd \Delta$, where $K$ is the kernel of the epimorphism $\theta: \Delta \longrightarrow \{ \pm 1 \}$ defined by the rule $x,y \longmapsto -1$, so that the inclusion $K \unlhd \Delta$ induces the function $\pi_1$.

It follows that the rule $\sigma_0, \sigma_1 \longmapsto -1$ indeed defines  a homomorphism $\rho: M \longrightarrow \{ \pm 1 \}$ and that, in fact, $\theta= \rho \circ \omega$. In turn this last identity implies that $\mathrm{Stab}_M(1) = \omega(\Gamma) < \omega(K)= \ker(\rho)$, as required.

 Conversely, let us assume the existence of a homomorphism 
 $$
 \rho:M  \longrightarrow  \{ \pm1\}: \sigma_0, \sigma_1  \longmapsto   -1
 $$
 such that ${\rm Stab}_{M}(1) \le \ker(\rho)$. Then there exists a commutative diagram 
   
   \centerline{
\xymatrix{  
\Delta \ar[r]^{\omega} \ar[rd]_{\theta}  &  M \ar[d]^{\rho} \\   & \{  \pm 1 \} }
}

\

\noindent 
where $\omega $ is the monodromy map and $\theta$ is defined by $\theta(x)=\theta(y)=-1$. It follows that
$$
\Gamma = \omega^{-1}(\mathrm{stab}_M(1)) \subset  \omega^{-1}(\ker \rho) = \ker \rho \circ \omega = \ker \theta = K
$$
and we have a  commutative diagram of morphisms of Riemann surfaces
  
  \centerline{
\xymatrix{ S = \mathbb{D}/ \Gamma  
 \ar[r]^{P} \ar[rd]_{\beta}  &  \mathbb{D}/ K  \ar[d]^{\pi} \\   & \mathbb{D}/ \Delta  }
}

\

The vertical arrow is a degree two regular covering of the Riemann sphere which ramifies over only two points, namely the quotient by the group $G= \Delta/K $. After suitable identifications $\mathbb{D}/ \Delta \simeq \widehat{\mathbb{C}}$ and $\mathbb{D}/ K \simeq \widehat{\mathbb{C}}$, the projection $\beta$ becomes the Belyi function associated with $\mathcal{D}$, the group $G$ corresponds to the group generated by the involution  $J(x)=1/x$, and the function $\pi$ corresponds to 
$\pi_{1}(z)= (1+z)^2/4z$. It follows that $\beta = \pi_{1} \circ P = (1+P)^2/4P$ and a direct computation gives 
$$\beta(1-\beta)= \displaystyle\left( \displaystyle\frac{i(P^2-1)}{4P}  \right)^{2}.$$
\end{proof}

\medskip
\begin{rema}Concerning the use of fiber products in the above proof the following facts should be mentioned.
\begin{enumerate}
\item  The Z-orientability of the dessin induced by the Belyi pair $(S,\beta)$ is equivalent to the reducibility of the fiber product 
$$
S \times_{(\beta, \pi_1)} \widehat{\mathbb{C}} = \left\{ (z,w) \in S \times \widehat{\mathbb{C}}: 
\ \beta(z)=\pi_1(w)
 \right\}.
$$
In the case when it is irreducible the pair $(S \times_{(\beta, \pi_1)} \widehat{\mathbb{C}}, Q_1 \circ \beta)$ is a Belyi pair corresponding to a  Z-orientable dessin covering the original non-Z-orientable dessin with degree two.

\item The dessin induced by the Belyi pair $(S,\beta)$ is Z-orientable if and only if
there exists a non-constant meromorphic function $P:S \to \widehat{\mathbb{C}}$, whose branch values are contained in the set $\{\infty,0,-1,1\}$, such that $\beta=\pi_{1} \circ P$. In this situation the Belyi pair $(S,\delta)$, where $\delta=P^2/(P^2 -1)$, defines the dessin induced by the dual bipartite map.

\end{enumerate}
\end{rema}

\medskip
\begin{rema}
Note that in the case of regular dessins we have $\mathrm{Stab}_M(1)=\{1\}$, and therefore the last condition on  $\rho$ in Proposition \ref{prop_zapponi} is trivially satisfied. 
\end{rema}

\medskip

Proposition \ref{prop_zapponi} has an obvious immediate consequence:

\medskip
\begin{coro} \label{coro:index2}
Let $\mathcal{D}$ be a dessin with monodromy group $M$. If $M$ does not contain an index two subgroup, then $\mathcal{D}$ is not Z-orientable. In particular, a dessin whose monodromy group is a simple group of order greater  than two is never Z-orientable.
\end{coro} 

\medskip 

As the monodromy group $M$ of a dessin is a two-generator group, the group $M/M^2$, where $M^2$ is the subgroup generated by the squares of its elements, is a two-generator group of exponent dividing $2$ and therefore isomorphic to  $1, \mathbb{Z}_2$ or $ \mathbb{Z}_2 \times \mathbb{Z}_2$. (All three cases arise, for instance when $M$ has odd order, or when $M$ is a dihedral group of order $2n$ for $n$ odd or even, respectively.) This fact will be used in Section \ref{tot}.

\medskip
\begin{coro} \label{le:msq}
Let ${\mathcal D}$ be a dessin with monodromy group $M$. If $M^{2}=M$, then ${\mathcal D}$ is not Z-orientable.
\end{coro}
\begin{proof}
Each homomorphism $\rho:M \to \{\pm 1\}$ factors through $M/M^{2}$.
\end{proof}

\medskip
\begin{exam} \label{example}
An interesting property of Z-orientability is that, unlike the monodromy group and other more sophisticated  
invariants, it can often be checked with a simple glance at the dessin. For instance the two dessins in  Figure~\ref{Fig:example} have the same passport $(4^2; 4^2; 4^2)$ and isomorphic monodromy groups (as can be checked using GAP \cite{GAP}; see below) but they are not Galois conjugate because while the left-hand dessin $\mathcal{D}$ is not Z-orientable (edges $1, 3, 6$ and $8$ are monofacial; see the first paragraph of Section \ref{Sec:orientable}), the right-hand dessin $\widetilde{\mathcal{D}}$ is, as the two faces miraculously share each of the eight edges. If $(S, \beta)$ is the Belyi pair corresponding to the dessin ${\mathcal D}$ of this example,  whose permutation representation is given by $\sigma_0=(1,2,3,4)(5,6,7,8)$, $\sigma_1=(1,5,2,6)(3,7,4,8) $, then $\widetilde{\mathcal{D}}$, with permutation representation $\widetilde{\sigma_0}=\sigma_{\infty}=(1,5,4,6)(2,8,3,7)$, $\widetilde{\sigma_1}=\sigma_{1}$, 
 corresponds to $(S,1/ \beta)$. That is, $\widetilde{\mathcal{D}}$ is twist-related to $\mathcal{D}$. 
These two dessins are of genus $g=2$, the monodromy group is isomorphic to $(({\mathbb Z}_{4}\times{\mathbb Z}_{2})\rtimes{\mathbb Z}_{2})\rtimes{\mathbb Z}_{2}$ and, moreover, the $M$-stabilizer of $1$ is $H=\langle (2,4)(6,8), (5,7)(6,8)\rangle \cong {\mathbb Z}_{2} \times {\mathbb Z}_{2}$.

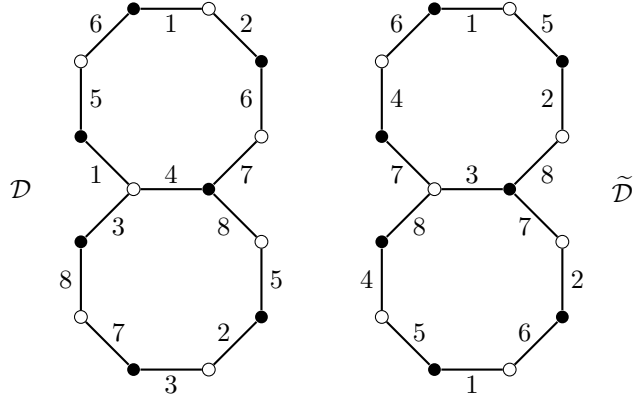
\begin{figure}[h!]
\begin{center}
\begin{tikzpicture}[scale=0.1, inner sep=0.6mm]

\node (A) at (-10,0) [shape=circle, draw] {};
\node (B) at (0,0) [shape=circle, fill=black] {};
\node (C) at (7,7) [shape=circle, draw] {};
\node (D) at (7,17) [shape=circle, fill=black] {};
\node (E) at (0,24) [shape=circle, draw] {};
\node (F) at (-10,24) [shape=circle, fill=black] {};
\node (G) at (-17,17) [shape=circle, draw] {};
\node (H) at (-17,7) [shape=circle, fill=black] {};
\node (I) at (-17,-7) [shape=circle, fill=black] {};
\node (J) at (-17,-17) [shape=circle, draw] {};
\node (K) at (-10,-24) [shape=circle, fill=black] {};
\node (L) at (0,-24) [shape=circle, draw] {};
\node (M) at (7,-17) [shape=circle, fill=black] {};
\node (N) at (7,-7) [shape=circle, draw] {};

\draw [thick] (A) to (B) to (C) to (D) to (E) to (F) to (G) to (H) to (A);
\draw [thick] (A) to (I) to (J) to (K) to (L) to (M) to (N) to (B);

\node at (-5,2) {$4$};
\node at (5,2) {$7$};
\node at (5,12) {$6$};
\node at (5,22) {$2$};
\node at (-5,22) {$1$};
\node at (-15,22) {$6$};
\node at (-15,12) {$5$};
\node at (-15,2) {$1$};

\node at (-12,-5) {$3$};
\node at (-19,-12) {$8$};
\node at (-12,-19) {$7$};
\node at (-5,-26) {$3$};
\node at (2,-19) {$2$};
\node at (9,-12) {$5$};
\node at (2,-5) {$8$};

\node at (-25,0) {$\mathcal D$};


\node (a) at (30,0) [shape=circle, draw] {};
\node (b) at (40,0) [shape=circle, fill=black] {};
\node (c) at (47,7) [shape=circle, draw] {};
\node (d) at (47,17) [shape=circle, fill=black] {};
\node (e) at (40,24) [shape=circle, draw] {};
\node (f) at (30,24) [shape=circle, fill=black] {};
\node (g) at (23,17) [shape=circle, draw] {};
\node (h) at (23,7) [shape=circle, fill=black] {};
\node (i) at (23,-7) [shape=circle, fill=black] {};
\node (j) at (23,-17) [shape=circle, draw] {};
\node (k) at (30,-24) [shape=circle, fill=black] {};
\node (l) at (40,-24) [shape=circle, draw] {};
\node (m) at (47,-17) [shape=circle, fill=black] {};
\node (n) at (47,-7) [shape=circle, draw] {};

\draw [thick] (a) to (b) to (c) to (d) to (e) to (f) to (g) to (h) to (a);
\draw [thick] (a) to (i) to (j) to(k) to (l) to (m) to (n) to (b);

\node at (35,2) {$3$};
\node at (45,2) {$8$};
\node at (45,12) {$2$};
\node at (45,22) {$5$};
\node at (35,22) {$1$};
\node at (25,22) {$6$};
\node at (25,12) {$4$};
\node at (25,2) {$7$};

\node at (28,-5) {$8$};
\node at (21,-12) {$4$};
\node at (28,-19) {$5$};
\node at (35,-26) {$1$};
\node at (42,-19) {$6$};
\node at (49,-12) {$2$};
\node at (42,-5) {$7$};

\node at (55,0) {$\widetilde{\mathcal D}$};

\end{tikzpicture}
\end{center}
\caption{A non-Z-orientable dessin $\mathcal{D}$ and a Z-orientable dessin $\widetilde{\mathcal{D}}$ with the same passport and the same monodromy group.}
\label{Fig:example}
\end{figure}
\end{exam}

By Theorem~\ref{th:squarecharact}, Z-orientability of dessins of genus~$0$ is encoded in the passport, as  the necessary condition of possessing only even degree vertices is also sufficient to ensure Z-orientability in this case (when all the vertices have even degree all the zeroes and poles of the rational function $\beta(1- \beta)$ have even order, hence $\beta(1- \beta)$ is a square). 
On the other hand, as we have seen in the previous example, the Z-orientability of a dessin cannot be determined in general by its passport and its monodromy group.  


\section{An alternative point of view}\label{alternative}

Here we outline an alternative approach to the Z-orientability properties of dessins, relying more on the concepts and language of group theory and combinatorics. As explained in Section \ref{Sec:dessins}, a dessin $\mathcal D$ may be identified with
a $5$-tuple of the form $(M, H, m_0, m_1, m_{\infty})$, where $M$ is a finite group with generators $m_0, m_1$ and $m_{\infty}$ satisfying $m_0m_1m_{\infty}=1$, and $H$ is a subgroup of $M$ with trivial core. 

We will now give some alternative characterizations of $Z$-orientability, implicitly using this correspondence to translate concepts about dessins from one language to another. The following result completes the proof of Theorem~\ref{mainthm}. Recall that a permutation group is {\em imprimitive} if it preserves a non-trivial equivalence relation, in which case it permutes the equivalence classes (called blocks of imprimitivity, or simply blocks).

\medskip

\begin{prop}\label{equivalent}
For any dessin $\mathcal D$, the following are equivalent (with the usual notation):
\begin{enumerate}

\item $\mathcal D$ is Z-orientable.

\item The faces of $\mathcal D$ can be labelled $+$ or $-$ so that consecutive faces around each vertex have different labels.

\item The edges of $\mathcal D$ can be labelled $+$ or $-$ so that consecutive edges around each vertex have different labels.

\item $M$ acts imprimitively on the edges of $\mathcal D$, with two blocks transposed by $m_0$ and $m_1$.

\item $M$ has a subgroup $N$ of index $2$ containing the element $m_{\infty}$ and the stabiliser $H$ of each edge.

\item $\mathcal D$ covers the unique dessin of degree $2$ and type $(2,2,1)$, shown in Figure~\ref{degree2dessin}.
\end{enumerate}
\end{prop}

\begin{figure}[h!]
\begin{center}
\begin{tikzpicture}[scale=0.4, inner sep=0.8mm]

\node (a) at (0,0) [shape=circle, draw] {};
\node (b) at (5,0) [shape=circle, fill=black] {};

\draw [thick] (5,0) arc (45:134:3.45);
\draw [thick] (5,0) arc (-45:-134:3.45);

\end{tikzpicture}
\end{center}
\caption{The unique dessin of degree $2$ and type $(2,2,1)$} 
\label{degree2dessin}
\end{figure}
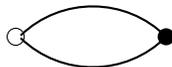

(There are similar equivalent conditions, with obvious slight modifications, for the Z-orientability of the dessins which are twist-related to $\mathcal D$.)

\medskip

\begin{proof}
The proof that each statement is equivalent to the next is straightforward, so we will just outline the basic ideas. For instance, (2) is essentially a restatement of the definition of Z-orientability. To show that (2) implies (3), one can transfer labels from faces to edges by following the orientation around the white vertices (or equivalently following the reverse orientation around the black vertices), and giving the sign on each face to the next edge; the converse is similar. For the equivalence of (3) and (4), having the same label is the required relation on edges. The subgroup $N$ in (5) is the kernel of the action on blocks, and the two blocks in (4) are its orbits. The inclusion $H\le N$ in (5) is equivalent to the covering in (6).
\end{proof}

\medskip

The subgroups $H\le N\le M$ appearing here are the images in $M$ of the subgroups $\Gamma\le K\le \Delta$ under the epimorphism $\omega:\Delta\to M$ defined earlier. As an application of this result, Proposition~\ref{prop_zapponi} is an immediate consequence of (1) being equivalent to (5). Similarly, the Galois invariance results in Section 3 follow immediately from the equivalence of (1) and (6), together with the fact that the absolute Galois group preserves degree, type and coverings. The two blocks for the Z-orientable dessin $\widetilde{\mathcal{D}}$ in Example~\ref{example} are the sets $\{1,2,3,4\}$ and $\{5,6,7,8\}$. The dessin in Figure~\ref{degree2dessin} corresponds to the subgroup of the free group of rank $2$ consisting of the words of even length in the standard generators. If we place this dessin in $\widehat{\mathbb C}$ with its white and black vertices at $-1$ and $1$, and its face-centres at $0$ and $\infty$, it corresponds to the Belyi function $\pi_1:\widehat{\mathbb C}\to\widehat{\mathbb C}$ given by $z\mapsto (z+1)^2/4z$ in the proof of Proposition~\ref{prop_zapponi}. A quarter turn, with white and black vertices moved to $-i$ and $i$, gives the Belyi function $\pi_2: z\mapsto (z^2-1)/4iz$ in the same proof.


\section{The twist-invariant orientability type}\label{tot}

One might want to avoid issues related to twists (see Example \ref{example}) by slightly modifying the notion of Z-orientability in the following way.

\medskip

\begin{defi}
The \emph{twist-invariant orientability type} of a dessin $\mathcal{D}$ is the number of dessins  in the twist-related triple $\{\mathcal{D}, \mathcal{D}', \mathcal{D}'' \}$ that are Z-orientable. We denote it by ${\rm tot}(\mathcal{D})$. Note that it is sufficient to restrict attention to $\mathcal D$ and the dessins ${\mathcal D}'$ and ${\mathcal D}''$ obtained from it by cyclic permutation of the generators $\sigma_0, \sigma_1$ and $\sigma_{\infty}$: the other three twist-related dessins are obtained from these by transposing vertex colours, which has no effect on Z-orientability.
\end{defi}

\medskip
\begin{prop}\label{totinvariant}
The twist-invariant orientability type is a Galois invariant, and ${\rm tot}(\mathcal{D}) \neq 2$ for every dessin $\mathcal{D}$.
\end{prop}
\begin{proof}
Using Theorem \ref{th:squarecharact}, one can easily show that if two of the three Belyi functions $\beta$, $(\beta-1)/\beta$, $1/(1-\beta)$ correspond to a Z-orientable dessin then the third one also does. Note that the three functions
$$
\beta(1-\beta), \quad \displaystyle\frac{\beta-1}{\beta} \left( 1- \displaystyle\frac{\beta-1}{\beta} \right)= \displaystyle\frac{\beta-1}{\beta^2}, \quad \left( \displaystyle\frac{1}{1-\beta} \right) \left( 1- \displaystyle\frac{1}{1-\beta}  \right) = \displaystyle\frac{-\beta}{(1-\beta)^2}
$$
have product equal to $1$, so if two of them are squares of meromorphic functions then so is the third.
Obviously, Theorem~\ref{th:squarecharact} also shows that ${\rm tot}$ is a Galois invariant.
\end{proof}

\medskip
\begin{rema}
Note that for a dessin ${\mathcal D}$ with at least one face of odd degree, ${\rm tot}$ is exactly the same invariant as plain Z-orientability. In fact, if one face has odd degree, the dessins ${\mathcal D}'$ and ${\mathcal D}''$ will each have a vertex of odd degree, so they cannot be Z-orientable. It is an immediate consequence of Corollary \ref{coro:index2} that for regular dessins ${\rm tot}=0$ if and only if the monodromy group does not contain a subgroup of index 2. Also, we may observe that ${\rm tot}=3$ is equivalent to the condition that $\beta$ and $1-\beta$ are the squares of meromorphic functions (by Theorem \ref{th:squarecharact} and the proof of Proposition \ref{totinvariant}), which can be restated as the existence of two meromorphic functions $f$ and $g$ on $S$ such that $f^{2}+g^{2}=1$ and $\beta=f^{2}$.  
\end{rema}

\medskip

In the following, we provide necessary and sufficient conditions, in terms of the monodromy group, for the Z-orientability properties of a dessin $\mathcal D$ and its twist-related dessins ${\mathcal D}'$ and ${\mathcal D}''$.

\medskip
\begin{theo}\label{teo15}
Let ${\mathcal D}$ be a dessin with monodromy group $M=\langle \sigma_{0},\sigma_{1}\rangle$ and let $H$ be the stabiliser in $M$ of a point.
\begin{enumerate}
\item If $M=M^{2}$, then ${\mathcal D}$ is not Z-orientable.

\item If $M/M^{2}\cong {\mathbb Z}_{2}$, then exactly one of the following holds:
$$\left\{\begin{array}{l}
{\mathcal D} \; \mbox{is Z-orientable, which occurs when }  H\le M^{2} \mbox{ and } \sigma_{\infty} \in M^{2}.\\
{\mathcal D}' \; \mbox{is Z-orientable, which occurs when } H\le M^{2} \mbox{ and } \sigma_{0} \in M^{2}.\\
{\mathcal D}'' \; \mbox{is Z-orientable, which occurs when }  H\le M^{2} \mbox{ and } \sigma_{1} \in M^{2}.\\
{\mathcal D}, {\mathcal D}', {\mathcal D}'' \; \mbox{are all non-Z-orientable, which occurs when $H\not\le M^{2}$.} 
\end{array}
\right.
$$

\item If $M/M^{2}\cong {\mathbb Z}_{2} \times {\mathbb Z}_{2}$, then
$$\left\{\begin{array}{l}
{\mathcal D}\; \mbox{is Z-orientable  if and only if}\; H\le \langle M^{2}, \sigma_{\infty}\rangle.\\
{\mathcal D}'\; \mbox{is Z-orientable if and only if}\; H\le \langle M^{2}, \sigma_{0}\rangle.\\
{\mathcal D}''\; \mbox{is Z-orientable if and only if}\; H\le \langle M^{2}, \sigma_{1}\rangle.\\
\end{array}
\right.
$$

\end{enumerate}
\end{theo}
\begin{proof}
Part (1) is just Corollary \ref{le:msq}.

In Case (2) there is a unique surjective homomorphism $\rho:M \to \{\pm 1\}$ whose kernel is $M^{2}$. There is exactly one pair $(a,b) \in 
\{(\sigma_{0},\sigma_{1}), (\sigma_{1},\sigma_{\infty}), (\sigma_{\infty},\sigma_{0})\}$ with the property that $\rho(a)=\rho(b)=-1$. Let $\widetilde{\mathcal D} \in \{{\mathcal D},{\mathcal D}', {\mathcal D}''\}$ be the dessin with permutation representation pair $(a,b)$. Then, by Zapponi's criterion, the other two dessins are not Z-orientable, and the dessin $\widetilde{\mathcal D}$ is Z-orientable if and only if $H\le M^{2}$.

In Case (3) we have that $\sigma_{0},\sigma_{1},\sigma_{\infty} \notin M^{2}$. There are exactly three surjective homomorphisms 
 $\overline{\rho}_{0}, \overline{\rho}_{1}, \overline{\rho}_{\infty}:M/M^{2} \to \{\pm 1\}$, where their kernels are, respectively, the cyclic groups generated by $\sigma_{0} M^{2}$, $\sigma_{1} M^{2}$ and $\sigma_{\infty} M^{2}$. Then, there are exactly three surjective homomorphisms $\rho_{j}=\overline{\rho}_{j} \circ \pi$, where $\pi:M \to M/M^{2}$ is the natural quotient homomorphism. As the kernels of these three homomorphisms are, respectively, $\langle M^{2}, \sigma_{0}\rangle$, $\langle M^{2}, \sigma_{1}\rangle$ and $\langle M^{2}, \sigma_{\infty}\rangle$, the desired result follows from Zapponi's criterion.
\end{proof}

\medskip
\begin{exam}\label{examplectd}
Let $\mathcal D$ and $\widetilde{\mathcal D}$ be the dessins in Example~\ref{example} and let $M$ be their common monodromy group. Since $M$ is a non-cyclic $2$-group we have $M/M^2\cong{\mathbb Z}_2\times{\mathbb Z}_2$,  so Case~(3) of Theorem~\ref{teo15} applies. We saw earlier that $\mathcal D$ is not Z-orientable, whereas $\widetilde{\mathcal D}$ is. Since ${\mathcal D}'$ differs from $\widetilde{\mathcal D}$ only by transposition of vertex-colours, it is also Z-orientable, so $1\le {\rm tot}({\mathcal D})<3$. It follows from Proposition~\ref{totinvariant} that ${\rm tot}({\mathcal D})=1$, so ${\mathcal D}''$ is not Z-orientable. This is confirmed  by checking that  $H\le \langle M^{2}, \sigma_j\rangle$ only for $j=0$. (These results, obtained by hand, were verified by using GAP.)
\end{exam}

\medskip

As a consequence of Theorem~\ref{teo15}, we obtain the following.

\medskip
\begin{coro}\label{teo16}
Let ${\mathcal D}$ be a dessin with monodromy group $M$ and let $H$ be the stabiliser in $M$ of a point. 
If $H\le M^{2}$, then
$${\rm tot}({\mathcal D})=\left\{\begin{array}{l}
0 \;\mbox{if and only if}\; M= M^{2}.\\
1 \;\mbox{if and only if}\; M/M^{2}\cong {\mathbb Z}_{2}.\\
3 \;\mbox{if and only if}\; M/M^{2}\cong {\mathbb Z}_{2} \times {\mathbb Z}_{2}.\\
\end{array}
\right.
$$
In particular, for regular dessins the ${\rm tot}$ invariant is completely determined by the monodromy group.
\end{coro}

\medskip
\begin{rema}
We note that Theorem~\ref{teo15} and Corollary~\ref{teo16} imply that the property of the stabiliser $H$ being a subgroup of $M^2$ is Galois invariant. If $M^2=M$ there is nothing to prove. If $M/M^2\cong{\mathbb Z}_2$ then $H\le M^2$ if and only if ${\rm tot}({\mathcal D})=1$, which is a Galois invariant property. Similarly, if $M/M^2\cong{\mathbb Z}_2\times{\mathbb Z}_2$ then $H\le M^2$ implies that ${\rm tot}({\mathcal D})=3$; the converse is also true since, by Theorem~\ref{teo15}, in this situation if ${\rm tot}(D)=3$ then $H \le \langle M^{2}, \sigma_0\rangle \cap \langle M^{2}, \sigma_1\rangle = M^{2}$. (Alternatively, this Galois invariance can be proved by interpreting $H$ as the covering group of the minimal regular cover of $\mathcal D$ (see~\cite[Appendix]{JS}), and using the fact that $M^2$ is a characteristic subgroup of $M$.)
\end{rema}

\medskip

As a consequence of the above, we also have the following fact.

\medskip
\begin{coro}\label{cubre}
Let ${\mathcal D}=(M,H,m_{0},m_{1},m_{\infty})$ be a dessin such that $H$ is not a subgroup of $M^{2}$. Let us consider the new dessin (having the same abstract monodromy group) given by the tuple 
${\mathcal D}_{0}=(M,H_{0}=H \cap M^{2},m_{0},m_{1},m_{\infty})$.
\begin{enumerate}
\item If $M/M^{2}\cong {\mathbb Z}_{2}$, then ${\rm tot}({\mathcal D}_{0})=1$, ${\rm tot}({\mathcal D})=0$ and ${\mathcal D}_{0}$ is a two-fold cover of ${\mathcal D}$.

\item If $M/M^{2}\cong {\mathbb Z}_{2}\times{\mathbb Z}_{2}$, then ${\rm tot}({\mathcal D}_{0})=3$, ${\rm tot}({\mathcal D})\in \{0,1\}$ and ${\mathcal D}_{0}$ is a two-fold or a ${\mathbb Z}_{2}\times{\mathbb Z}_{2}$-fold cover of ${\mathcal D}$.

\end{enumerate}
\end{coro}

\medskip
\begin{exam}
For an example of Case~(1), let $M$ be the dihedral group
\[\langle a, b, c\mid a^2=b^2=c^n=abc=1\rangle\]
of order $2n$ for some odd $n\ge 3$, acting naturally with stabiliser $H=\langle a\rangle$. Then $M^2=\langle c\rangle$, a subgroup of index $2$ which does not contain $H$. The dessin $\mathcal D$ of type $(2,2,n)$ corresponding to the generating triple $(a,b,c)$ is a path of length $n$ in the sphere, with alternating white and black vertices, shown on the left in Figure~\ref{dihedral}. The twisted dessin ${\mathcal D}'$, corresponding to the triple $(b,c,a)$, is shown on the right, and ${\mathcal D}''$ is obtained from ${\mathcal D}'$ by transposing the vertex colours. Clearly none of these dessins is Z-orientable, so ${\rm tot}({\mathcal D})=0$. We have $H_0=H\cap M^{2}=\{1\}$, so ${\mathcal D}_0$ is a regular dessin of type $(2,2,n)$; in fact it is the unique regular dessin of this type, consisting of a circuit of length $2n$ with alternating white and black vertices, embedded in the sphere. It is obviously Z-orientable, whereas $({\mathcal D}_0)'$ and  $({\mathcal D}_0)''$, containing vertices of odd degree $n$, are not, so ${\rm tot}({\mathcal D}_0)=1$.

\begin{figure}[h!]
\begin{center}
\begin{tikzpicture}[scale=0.12, inner sep=0.8mm]

\node (a) at (-30,0) [shape=circle, fill=black] {};
\node (b) at (-25,0) [shape=circle, draw] {};
\node (c) at (-20,0) [shape=circle, fill=black] {};
\node (d) at (-10,0) [shape=circle, draw] {};
\node (e) at (-5,0) [shape=circle, fill=black] {};
\node (f) at (0,0) [shape=circle, draw] {};

\draw [thick] (a) to (b) to (c) to (-17.5,0);
\draw [thick] (f) to (e) to (d) to (-12.5,0);
\draw [thick, dotted] (-16,0) to (-14,0);


\node (A) at (10,0) [shape=circle, fill=black] {};
\node (B) at (15,0) [shape=circle, draw] {};
\node (C) at (20,0) [shape=circle, draw] {};
\node (D) at (25,0) [shape=circle, draw] {};
\node (E) at (35,0) [shape=circle, draw] {};

\draw [thick] (A) to (B);
\draw [thick, dotted] (28.5,0) to (31.5,0);
\draw [thick] (19.65,0.35) arc (60:120:9.2);
\draw [thick] (19.65,-0.35) arc (-60:-120:9.2);
\draw [thick] (24.65,0.35) arc (50:130:11.35);
\draw [thick] (24.65,-0.35) arc (-50:-130:11.35);
\draw [thick] (34.65,0.35) arc (40:140:16.2);
\draw [thick] (34.65,-0.35) arc (-40:-140:16.2);

\end{tikzpicture}
\end{center}
\caption{The dessins $\mathcal D$ and ${\mathcal D}'$} 
\label{dihedral}
\end{figure}
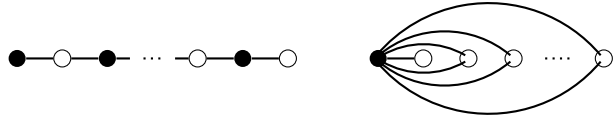

For an example of Case (2), consider the non-Z-orientable uniform dessin ${\mathcal D}$ in Example~\ref{example}. In this case $H$, isomorphic to ${\mathbb Z}_{2} \times {\mathbb Z}_{2}$, is not a subgroup of $M^{2}$, and $M/M^{2}\cong {\mathbb Z}_{2}\times{\mathbb Z}_{2}$. As the dessin ${\mathcal D}'$ differs from $\widetilde{\mathcal D}$ in that example only by transposition of the vertex colours, and as it is Z-orientable, we have ${\rm tot}({\mathcal D})={\rm tot}(\widetilde{\mathcal D})=1$. As in this case $H_{0}=H \cap M^{2} \cong {\mathbb Z}_{2}$ has index two in $H$, the dessins ${\mathcal D}_0$ and $\widetilde{\mathcal D}_0$ give two-fold (unbranched) coverings of the  dessins $\mathcal D$ and $\widetilde{\mathcal D}$, respectively. These are uniform but non-regular dessins of type $(4,4,4)$, passport $(4^4;4^4;4^4)$ and genus $3$. The dessin ${\mathcal D}_0$, with permutation representation
\[\sigma_0=(1,2',3,4')(1',2,3',4)(5,6',7,8')(5',6,7',8),\]
\[\sigma_1=(1,5',2,6')(1',5,2',6)(3,7',4,8')(3',7,4',8),\]
is shown in Figure~\ref{D_0}; the covering ${\mathcal D}_0\to{\mathcal D}$ is given by the mapping $e, e'\mapsto e$ of edge-labels for $e=1,\ldots, 8$. By Theorem~\ref{mainthm}(3), the partition of edges into those labelled $e$ or $e'$ shows that ${\mathcal D}_0$ is Z-orientable.
\end{exam}

\medskip

\begin{figure}[h!]
\begin{center}
\begin{tikzpicture}[scale=0.22, inner sep=0.8mm]

\node (a) at (0,2) [shape=circle, fill=black] {};
\node (a') at (0,-2) [shape=circle, draw] {};
\node (b) at (2.8,4.8) [shape=circle, draw] {};
\node (b') at (2.8,-4.8) [shape=circle, fill=black] {};
\node (c) at (6.8,4.8) [shape=circle, fill=black] {};
\node (c') at (6.8,-4.8) [shape=circle, draw] {};
\node (d) at (9.6,2) [shape=circle, draw] {};
\node (d') at (9.6,-2) [shape=circle, fill=black] {};
\node (e) at (12.4,4.8) [shape=circle, fill=black] {};
\node (e') at (12.4,-4.8) [shape=circle, draw] {};
\node (f) at (16.4,4.8) [shape=circle, draw] {};
\node (f') at (16.4,-4.8) [shape=circle, fill=black] {};
\node (g) at (19.2,2) [shape=circle, fill=black] {};
\node (g') at (19.2,-2) [shape=circle, draw] {};
\node (h) at (22,4.8) [shape=circle, draw] {};
\node (h') at (22,-4.8) [shape=circle, fill=black] {};
\node (i) at (26,4.8) [shape=circle, fill=black] {};
\node (i') at (26,-4.8) [shape=circle, draw] {};
\node (j) at (28.8,2) [shape=circle, draw] {};
\node (j') at (28.8,-2) [shape=circle, fill=black] {};
\node (k) at (31.6,4.8) [shape=circle, fill=black] {};
\node (k') at (31.6,-4.8) [shape=circle, draw] {};
\node (l) at (35.6,4.8) [shape=circle, draw] {};
\node (l') at (35.6,-4.8) [shape=circle, fill=black] {};
\node (m) at (38.4,2) [shape=circle, fill=black] {};
\node (m') at (38.4,-2) [shape=circle, draw] {};

\draw [thick] (a) to (b) to (c) to (d) to (e) to (f) to (g) to (h) to (i) to (j) to (k) to (l) to (m);
\draw [thick] (a') to (b') to (c') to (d') to (e') to (f') to (g') to (h') to (i') to (j') to (k') to (l') to (m');
\draw [thick] (a) to (a');
\draw [thick] (d) to (d');
\draw [thick] (g) to (g');
\draw [thick] (j) to (j');
\draw [thick] (m) to (m');

\node at (-1,0) {$1$};
\node at (39.4,0) {$1$};
\node at (27.8,0) {$2$};
\node at (18.2,0) {$3$};
\node at (8.6,0) {$4$};
\node at (14.4,5.8) {$5$};
\node at (4.8,-5.8) {$5$};
\node at (4.8,5.8) {$6$};
\node at (33.6,-5.8) {$6$};
\node at (33.6,5.8) {$7$};
\node at (24,-5.8) {$7$};
\node at (24,5.8) {$8$};
\node at (14.4,-5.8) {$8$};

\node at (1,4.2) {$2'$};
\node at (10.6,4.2) {$8'$};
\node at (20.2,4.2) {$4'$};
\node at (29.8,4.2) {$6'$};

\node at (0.9,-4.2) {$6'$};
\node at (10.5,-4.2) {$3'$};
\node at (20.1,-4.2) {$8'$};
\node at (29.7,-4.2) {$1'$};

\node at (8.7,4.2) {$7'$};
\node at (18.3,4.2) {$2'$};
\node at (27.9,4.2) {$5'$};
\node at (37.5,4.2) {$4'$};

\node at (8.7,-4.2) {$1'$};
\node at (18.3,-4.2) {$7'$};
\node at (27.9,-4.2) {$3'$};
\node at (37.5,-4.2) {$5'$};

\end{tikzpicture}
\end{center}
\caption{The dessin ${\mathcal D}_0$ of genus $3$} 
\label{D_0}
\end{figure}
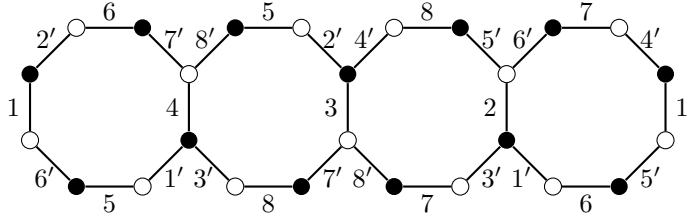


\section{Proof of Theorem~\ref{independent}}
In this last section we will prove Theorem~\ref{independent} by constructing an infinite set of instances of four uniform but non-regular dessins with the same monodromy group and generating triple (as an abstract group), with isomorphic point stabilisers and automorphism groups and with the same passport, but with different Z-orientability properties, so that no two of them are conjugate under the absolute Galois group. This shows that in some circumstances Z-orientability can be used to distinguish dessins in separate Galois orbits when these other properties are insufficient for this purpose. We will describe all our dessins as tuples $(M,H,m_{0},m_{1},m_{\infty})$, as explained in Section~\ref{Sec:dessins}.

The construction starts by taking $M=S_n\times S_n$ for some $n\ge 3$. As $A_n$ has index $2$ in $S_n$ we have $S_n^2\le A_n$, and as $A_n$ is generated by $3$-cycles these subgroups are equal, so $M^{2}=A_{n} \times A_{n}$ and $M/M^{2} \cong {\mathbb Z}_{2}\times {\mathbb Z}_2$. (In fact, it can be seen that $M^{2}$ is the commutator subgroup $M'$ of $M$.)

Let us consider the elements $m_0=(y,z), m_1=(z,x), m_{\infty}=(x,y) \in M$, where
\[x=(1,2,\ldots, n), \quad y=(1,2), \quad  z=(xy)^{-1}=(n,n-1,\dots, 2)\in S_{n}. \]

\medskip
\begin{lemm}\label{generate}
The elements $m_0$, $m_1$ and $m_{\infty}$ generate $M$, and satisfy $m_{0} m_{1} m_{\infty}=1$. They have order $n-1$, $n(n-1)$ and $2n$ if $n$ is odd, and order $2(n-1)$, $n(n-1)$ and $n$ if $n$ is even.
\end{lemm} 
\begin{proof}
Note that any two of $x, y$ and $z$ generate $S_n$, and $xyz=1$, so the elements $m_j$ satisfy $m_{0} m_{1} m_{\infty}=1$. Their orders follow from the fact that $x$, $y$ and $z$ have order $n$, $2$ and $n-1$.
Let $S$ denote the subgroup of $M$ which the elements $m_j$ generate. If $n$ is odd then $S$ contains $m_{1}^{n}=(z^n,x^n)=(z,1)$ and $m_{\infty}^{n+1}=(x^{n+1},y^{n+1})=(x,1)$, so $S$ contains the first direct factor $S_n\times\{1\}$ since $x$ and $z$ generate $S_n$. Similarly, since $m_{1}^{n-1}=(1,x^{-1})$ and $m_{\infty}^n=(1,y)$ we have $S\ge \{1\}\times S_n$, so $S=M$. The proof of this is similar if $n$ is even, using suitable powers of $m_0$ and $m_1$.
\end{proof}

\medskip

Let us define the following index two subgroups $M_{0}=\langle M^{2}, m_{0}\rangle$, $M_{1}=\langle M^{2},m_{1}\rangle$ and $M_{\infty}=\langle M^{2},m_{\infty}\rangle$ of $M$. Any two of them intersect in $M^{2}=A_n\times A_n$, and each $m_{j}$ lies in just one of them, namely $M_{j}$. If $n$ is odd then $M_j\; (j=0, 1, \infty)$ consists of the pairs $(g,h)$ for which $gh$, $h$ or $g$ is an even permutation, respectively; if $n$ is even the corresponding conditions are that $h$, $g$ or $gh$ are even. Theorem~\ref{teo15}(3) implies the following:

\medskip
\begin{lemm} \label{le:condition}
Let $H$ be a subgroup of $M$ with trivial core, and let ${\mathcal D}$ be the dessin $(M,H,m_{0},m_{1},m_{\infty})$. 
 Then the dessins $\mathcal{D}, \mathcal{D}', \mathcal{D}''$ are Z-orientable if and only if $H$ is a subgroup of $M_{\infty}$, $M_0$ or $M_1$ respectively. 
\end{lemm}

\medskip

We will now construct four subgroups $H_{i}$ ($i=0,1,2,3$) of $M$, each one with trivial core, and will apply this lemma to the dessins ${\mathcal D}_{i}=(M,H_{i},m_{0},m_{1},m_{\infty})$. The construction proceeds in several steps, as follows.

The rotation group $Q$ of a cube is isomorphic to $S_4$, through its faithful action on the four main diagonals, and its isometry group $G$ has the form $Q\times I$, where $I\;(\cong S_2)$ is generated by the antipodal isometry $i$. The faithful action of $G$ on the faces of the cube gives an embedding of $G$ in $S_6$, with $Q$ acting transitively. In addition to $Q$, there is another subgroup $P\cong S_4$ in $G$, also acting transitively on the faces: this is the subgroup $Q'\cup (Q\setminus Q')i$ preserving the two tetrahedra inscribed in the cube, and acting as the isometry group of each. (Here $Q'\;(\cong A_4)$ is the commutator subgroup of $Q$.) Although $P$ and $Q$ are isomorphic subgroups of $S_6$, they are not conjugate in $S_6$: for instance, their elements of order $4$ have cycle structures $2,4$ and $1,1,4$ respectively. Having index $2$ in $G$, both $P$ and $Q$ are normal in $G$, so their normalisers in $S_6$ contain $G$. They are not normal in $S_6$, and $G$ is a maximal subgroup of $S_6$ (it has index $15$, and $S_6$ has no subgroups of index $3$ or $5$), so they have the same normaliser $N_{S_6}(P)=N_{S_6}(Q)=G$. (Note that $P$ and $Q$ are not equivalent under the non-trivial outer automorphism of $S_6$, which  instead transposes them with intransitive subgroups of $S_6$.)

For any $n\ge 6$ let $A$ and $B$ be subgroups of $S_n$ isomorphic to $S_4$, acting as $P$ or $Q$ on a set of six points and fixing the remaining $n-6$ points. Thus $A$ and $B$ are isomorphic but non-conjugate subgroups of index $n!/4!$ in $S_n$. It is easy to check that $A \le A_n$, whereas $B\not\le A_n$: for example, an element of order $4$ in $A$ or $B$ is respectively even or odd. We have $N_{S_n}(A)=N_{S_n}(B)=G\times S_{n-6}$, with the direct factors arising from the points moved and fixed by $A$ and $B$, so $N_{S_n}(A)/A \cong N_{S_n}(B)/B \cong S_2\times S_{n-6}$.

As point stabilisers for the four dessins, define $H_0:=B\times B$, $H_1:=B\times A$, $H_2:=A\times B$, and $H_3:=A\times A$, mutually isomorphic but non-conjugate subgroups of index $N:=(n!/4!)^2$ in $M$. Note that $H_1$ and $H_2$ are each contained in a unique subgroup $M_{j}$ of index $2$ in $M$, namely $M_1$ and $M_{\infty}$ respectively if $n$ is odd, and $M_0$ and $M_1$ if $n$ is even, while $H_3$ is contained in all three of them and $H_0$ in none. Lemma~\ref{le:condition} now implies that the dessins ${\mathcal D}_{i}=(M,H_{i},m_{0},m_{1},m_{\infty})$ satisfy ${\rm tot}(\mathcal{D}_i)=0,1, 1$ and $3$ for $i=0,1,2$ and $3$. Moreover, if $n$ is odd then since $H_2\le M_{\infty}$ whereas $H_1\not\le M_{\infty}$, ${\mathcal D}_2$ is Z-orientable whereas ${\mathcal D}_1$ is not; similarly, if $n$ is even then ${\mathcal D}_1'$ is Z-orientable whereas ${\mathcal D}_2'$ is not. Thus these four dessins have different Z-orientability properties, so they lie in distinct Galois orbits.

We now consider the monodromy and automorphism groups of these dessins. For each $i=0, 1, 2, 3$, the natural action of $M$ on the cosets of the subgroup $H_i$ yields an embedding $\varphi_i : M \to \mathrm{Symm}(M/H_i) \cong S_N$  which allows us to regard $M$ as a transitive subgroup of $S_N$. For each $j=0, 1, \infty$ define $\sigma_j$ to be the permutation $\varphi_i(m_j)\in S_N$ induced by $m_j$, so that the triple $(\sigma_0, \sigma_1, \sigma_{\infty})$ determines a dessin $\mathcal{D}_i$ with monodromy group $M$.
As $H_i$ is the stabiliser of an edge, each dessin ${\mathcal D}_i$ has automorphism group
\[{\rm Aut}({\mathcal D}_{i}) \cong C_{S_N}(\varphi_i(M)) \cong N_M(H_i)/H_i\]
where $C$ and $N$ denote centraliser and normaliser (see~\cite[Corollary 2.1]{JW}). Now each $H_i$ has the form $X\times Y\le M=S_n\times S_n$, where $X, Y\in\{A, B\}$, so $N_M(H_i)=N_{S_n}(X)\times N_{S_n}(Y)$, and hence
\[N_M(H_i)/H_i\cong (N_{S_n}(X)/X)\times (N_{S_n}(Y)/Y).\]
We saw earlier that $N_{S_n}(A)/A \cong N_{S_n}(B)/B \cong S_2\times S_{n-6}$, so
\[{\rm Aut}({\mathcal D}_i)\cong (S_2\times S_{n-6})\times (S_2\times S_{n-6})\]
for each $i$.

By construction the four dessins ${\mathcal D}_i$ all have the same monodromy group and generating triple, and have isomorphic point stabilisers and automorphism groups, so it remains only to show that they have the same passport. This is an immediate consequence of the following:

\medskip
\begin{lemm}\label{uniform}
For each $n\ge 8$ the dessins ${\mathcal D}_i\; (i=0,\ldots, 3)$ are all uniform, having type $(n-1, n(n-1), 2n)$ or $(2(n-1), n(n-1), n)$ as $n$ is odd or even.
\end{lemm}
\begin{proof} It is sufficient to show that no non-identity power of any generator $m_{j}\;(j=0, 1, {\infty})$ of $M$ is conjugate to an element of any subgroup $H_i$, for in that case  the dessins ${\mathcal D}_i$ will all be uniform, with the stated type. This can be seen as follows: if in the cycle decomposition of one of the generators $\sigma_j=\varphi_i(m_j) \in S_N$ we had a cycle of length $k$ less than its order, then $\sigma_{j}^k$ would have to fix some point in $\{1, 2, \ldots , N \}$, and by the transitivity of $M$ would have to lie in a conjugate of the corresponding subgroup $H_i$. Now by inspection of the cube each non-identity element of $P$ or $Q$ fixes either two or none of the six points it permutes, so each non-identity element of $A$ or $B$ fixes either $n-4$ or $n-6$ points of the $n$ points permuted. On the other hand, a non-identity power of $x, y$ or $z$ fixes none, $n-2$, or $1$, so for $n\ge 8$ a non-identity power of $m_{j}$ cannot be conjugate to an element of any $H_i$.  
\end{proof}

\medskip

It follows that for odd $n\ge 9$ and for each dessin ${\mathcal D}_i$, $\sigma_0$ has $N/(n-1)$ cycles of length $(n-1)$, while $\sigma_1$ has $N/n(n-1)$ cycles of length $n(n-1)$ and $\sigma_{\infty}$ has $N/2n$ cycles of length $2n$, with a similar result for even $n\ge 8$. In particular, these four dessins all have the same passport, as claimed. (And hence, of course, they have the same genus.) This completes the proof of Theorem~\ref{independent}.

\medskip
\begin{rema}
Among the dessins ${\mathcal D}_i$ described above, the least degree $N$ arises for $n=8$, so that
\[N=\left({8!}/{4!}\right)^2=2,822,400.\]
In this case the dessins have type $(14, 56, 8)$ and genus $1,108,801$.
The authors regret that they are unable to provide drawings of these dessins.
\end{rema}

\medskip
\begin{rema}
The construction is also valid for $n=6$ and $n=7$. The dessins have the same properties as before, except that they are not uniform, and do not all have the same passport (or genus). For example, when $n=6$, so that the dessins have degree $900$ and type $(10, 30, 6)$:
\begin{itemize}
\item ${\mathcal D}_0$ has passport $(10^{90}; 10^{18}15^{24}30^{12}; 2^{90}3^{120}6^{60})$ and genus $244$;
\item ${\mathcal D}_1$ has passport $(10^{90}; 10^{18}30^{24}; 2^{90}3^{120}6^{60})$ and genus $250$;
\item ${\mathcal D}_2$ has passport $(10^{90}; 10^{18}15^{24}30^{12}; 2^{90}6^{120})$ and genus $274$;
\item ${\mathcal D}_3$ has passport $(10^{90}; 10^{18}30^{24}; 2^{90}6^{120})$ and genus $280$.
\end{itemize}
When $n=7$, so that the dessins have degree $44100$ and type $(6, 42, 14)$:
\begin{itemize}
\item ${\mathcal D}_0$ and ${\mathcal D}_1$ have passport $(3^{420}6^{7140}; 14^{90}21^{120}42^{960}; 14^{3150})$ and genus $16111$;
\item ${\mathcal D}_2$ and ${\mathcal D}_3$ have passport $(3^{420}6^{7140}; 14^{90}42^{1020}; 14^{3150})$ and genus $16141$.
\end{itemize}
These passports can all be found by using the fact that in any transitive permutation group $M$, with point stabiliser $H$, the number of fixed points of an element $m \in M$ is $|m^M\cap H|.|C_M(m)|/|H|$, where $m^M$ denotes the conjugacy class of $m$ in $M$ and $C_M(m)$ denotes its centraliser. By applying this to various powers of $m$ one can find how many cycles it has of each length dividing its order. Here we take $H=H_i$ for $i=0,\ldots, 3$, and $m=m_j$ for $j=0, 1, \infty$.
\end{rema}

\medskip
\begin{rema}
For each $n\ge 6$, the four dessins ${\mathcal D}_i$ all have the same minimal regular cover, namely the regular dessin with monodromy and automorphism group $M=S_n\times S_n$ and generating triple $(m_0, m_1, m_{\infty})$. This has degree $(n!)^2$, and is a regular $576$-sheeted covering, with covering group $H_i\cong S_4\times S_4$, of each ${\mathcal D}_i$. These coverings are smooth for $n\ge 8$, but branched for $n=6$ and $7$.

\begin{figure}[h!]
\begin{center}
\begin{tikzpicture}[scale=0.25, inner sep=0.8mm]

\node (a) at (-20,0) [shape=circle, fill=black] {};
\node (b) at (-15,0) [shape=circle, draw] {};
\node (c) at (-10,0) [shape=circle, fill=black] {};
\node (d) at (-17.5,4) [shape=circle, draw] {};
\node (e) at (-22.5,4) [shape=circle, draw] {};
\node (f) at (-22.5,-4) [shape=circle, draw] {};
\node (g) at (-17.5,-4) [shape=circle, draw] {};

\draw [thick] (a) to (b) to (c);
\draw [thick] (a) to (d);
\draw [thick] (a) to (e);
\draw [thick] (a) to (f);
\draw [thick] (a) to (g);

\draw [thick, dotted] (-22.5,2) arc (135:225:3);


\node (A) at (-4,0) [shape=circle, draw] {};
\node (B) at (4,0) [shape=circle, fill=black] {};
\node (C) at (9,0) [shape=circle, draw] {};
\draw [thick] (B) to (C);
\draw [thick, dotted] (0,0.5) to (0,-0.5);

\draw [thick] (4,0) arc (30:147:4.6);
\draw [thick] (4,0) arc (-30:-147:4.6);
\draw [thick] (4,0) arc (45:133:5.6);
\draw [thick] (4,0) arc (-45:-133:5.6);
\draw [thick] (4,0) arc (60:119:7.8);
\draw [thick] (4,0) arc (-60:-119:7.8);

\end{tikzpicture}
\end{center}
\caption{Two twist-related planar dessins of degree $n$} 
\label{planar}
\end{figure}
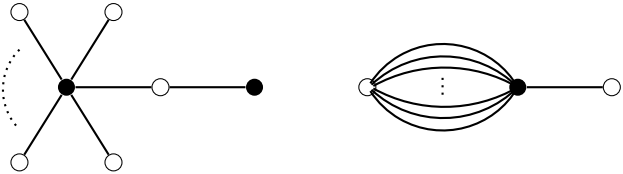

This regular dessin may seem rather complicated, but it is simply the minimal common regular cover of the pair of planar dessins of degree $n$ with monodromy group $S_n$, acting naturally, and generating triples $(y,z,x)$ and $(z,x,y)$ given by the first and second coordinates of the triple $(m_j)$ generating $M$. These twist-related dessins are shown on the left and right in Figure~\ref{planar}; the labelling of edges is obvious, and therefore omitted.
\end{rema}

\bigskip
\noindent
{\bf Aknowledgment:} The authors would like to thank the referee for her/his valuable comments, suggestions and corrections. 

\bigskip
\noindent
{\bf Thanks:} This work was partially supported by Project MTM2016-79497-P, Project SEV-2015-0554,  Project Fondecyt 1150003 and Project Anillo ACT1415 PIA-CONICYT

\end{document}